\documentclass[11pt,reqno]{amsart}
\usepackage{amssymb}
\usepackage{amsfonts}
\usepackage{mathrsfs}
\usepackage[all]{xy}

\def\bigast{\mathop{\hbox{\Large $\ast$}}}

\def\ilim{\varprojlim}

\def\inv{^{-1}}
\def\p{\varphi}

\def\e<{\leq _{E}}

\def\ov#1{\ensuremath{\overline {#1}}}

\def\til#1{\ensuremath{\widetilde {#1}}}

\def\malce{\mathop{\hbox{$\bigcirc$\kern-10pt\raise1pt
\hbox{\footnotesize$m$}\kern1.5pt}}}

\def\1sk{^{(1)}}

\def\to{\rightarrow}

\def\data{\ifcase\month\or January\or February \or March\or April\or May
\or June\or July\or August\or September\or October\or November \or
December\fi\space\number\day, \number\year}

\def\Thmname{Theorem}
\def\Propname{Proposition}
\def\Lemmaname{Lemma}
\def\Definitionname{Definition}
\newtheorem{Thm}{\Thmname}[subsection]
\newtheorem{Prop}[Thm]{\Propname}
\newtheorem{Lemma}[Thm]{\Lemmaname}
{\theoremstyle{definition}
\newtheorem{Def}[Thm]{\Definitionname}}
{\theoremstyle{remark}
\newtheorem{Rmk}[Thm]{Remark}}
\newtheorem{Cor}[Thm]{Corollary}

{\theoremstyle{remark}
\newtheorem{Case}{Case}}

\numberwithin{equation}{section}

\title{A Wreath Product Approach to Classical Subgroup Theorems}

\author{Luis Ribes\and Benjamin Steinberg}
\address{School of Mathematics and Statistics \\ Carleton University \\
1125 Colonel By Drive\\
Ottawa, Ontario  K1S 5B6 \\
Canada}
\thanks{The authors gratefully acknowledge the support of NSERC}
\email{lribes@math.carleton.ca\and bsteinbg@math.carleton.ca}
\date{November 28, 2008}

\begin{document}
\begin{abstract}
We provide elementary proofs of the Nielsen-Schreier Theorem and the Kurosh Subgroup Theorem via wreath products.  Our proofs are diagrammatic in nature and work simultaneously in the abstract and profinite categories.  A new proof that open subgroups of quasifree profinite groups are quasifree is also given.
\end{abstract}

\maketitle

\section{Introduction}
The purpose of this paper is to provide a conceptual framework for simple algebraic proofs of the classical subgroup theorems from combinatorial group theory: the Nielsen-Schreier Theorem and the Kurosh Theorem.  Our proof of the Nielsen-Schreier Theorem, for instance, could very easily be presented in a first course introducing free groups.  The fundamental idea is to exploit the functoriality of the wreath product in order to reduce these theorems to diagram chasing.  By removing as much as possible the combinatorics on words, we are able to present proofs that also work in the profinite category.  Traditionally, the subgroup theorems for profinite groups are obtained via a reduction to the abstract case; here we prove the abstract and profinite theorems simultaneously.

In addition to proving the classical subgroup theorems, we also give a very simple and natural proof of a result of the first author, Stevenson and Zalesskii~\cite{quasifree} on open subgroups of quasifree profinite groups.

The origins of our approach via wreath products lie in two sources: profinite group theory and profinite semigroup theory.   The genesis of the wreath product technique for subgroup theorems is~\cite{cossey}, where Cossey, Kegel and Kov{\'a}cs used wreath products to prove that closed subgroups of projective profinite groups are again projective.  The usual proofs of this result rely on the Nielsen-Schreier Theorem for abstract free groups or on cohomological techniques, see~\cite[Theorem 7.7.4]{RZbook} for example. The wreath product approach was further developed by Haran to study closed subgroups of free products of profinite groups~\cite{haran}.  Ershov~\cite{ershov} seems to be the first to have attempted to use wreath products to deal with subgroup theorems for discrete groups.  In particular, he gives a proof of the Kurosh Theorem using Haran's notion of a projective family.  However, his proof is not conceptually appealing since it follows this route of projective families, and moreover it relies on the Nielsen-Schreier Theorem, which normally should be deducible as a special case of the Kurosh Theorem.

The same wreath product techniques arose independently in the work of semigroup theorists investigating the structure of free profinite monoids.  Wreath products were first introduced into semigroup theory by Sch\"utz\-en\-ber\-ger~\cite{Schutzmonomial} and came to play a major role in the subject with the advent of the Krohn-Rhodes Theorem~\cite{PDT}, which definitively established the wreath product as the principal instrument for decomposing semigroups into simpler parts; see Eilenberg's book~\cite{Eilenberg} or~\cite{qtheor} for details.   There is no Nielsen-Schreier Theorem for free monoids; also cohomological techniques do not work well for semigroups because the Eckmann-Shapiro Lemma fails in this context.  Semigroup theorists were then naturally led to the wreath product to prove structural results about free profinite monoids.   Margolis, Sapir and Weil~\cite{MSWj} first exploited this technique in order to show that those finitely generated clopen submonoids of a free profinite monoid that have any chance to be free are indeed free; this was extended to the non-finitely generated case by Almeida and the second author~\cite{clopen}.  Rhodes and the second author rediscovered the proof from~\cite{cossey} that closed subgroups of projective profinite groups are projective and used an analogous argument to establish that closed subgroups of free profinite monoids are projective profinite groups~\cite{projective}; see also~\cite{minimalideal} where a similar idea was used.

We soon came to realize that the theorems for abstract groups should also be amenable to these techniques, leading to the current paper.  The paper is organized as follows.  The first section sets up our notation for wreath products and establishes the basic functorial properties of this construction.  Next we turn to the Nielsen-Schreier Theorem, which is proved for abstract groups and then adapted to profinite groups.    The Nielsen-Schreier Theorem is followed up by the Kurosh Theorem, which is the most technical part of the paper.  The paper closes with a proof that open subgroups of quasifree profinite groups are again quasifree.  The aim of this paper is to be a self-contained and elementary exposition, so many well-known results are included.

\section{Notation and Conventions}\label{section1}
If $K$ and $L$ are groups, $K\le L$ indicates that $K$ is a subgroup of $L$.
Composition of maps in this paper is always assumed to be right-to-left, except when dealing with permutations in a symmetric group
$S_\Sigma$, which we multiply left-to-right.  If $x, y\in L$, we define $x^y= y^{-1}xy$   and $K^y= y^{-1}Ky$.  The inner automorphism $\mathrm{inn}_y $ of $L$ determined by $y$ is the automorphism $x\mapsto yxy^{-1}$  $(x\in L)$.

\subsection{The Semidirect Product}
Recall that a group $G$ is said to act  on a group $R$ on the left, if
there exists a homomorphism
$\alpha\colon  G\rightarrow {\rm Aut}(R)$
denoted by $g\mapsto \alpha_g$ ($g\in G$).

Equivalently, $G$ acts on $R$ on the left if there
is a function $G\times R\rightarrow R$ denoted
by
$(x,r) \mapsto {}^x\!{r}$, such that
\begin{itemize}
\item [(a)] ${}^1\!{r} =r$, $\forall  r\in R$,

\item [(b)] ${}^{xy}\!{r}=  {}^x\!{({}^y\!{r})} $, $\forall  r\in R$,
$x,y\in G$,

\item [(c)]  ${}^x\!{(r_1r_2)}=  {}^x\!{r_1}{}^x\!{r_2}$, $\forall  r_1,
r_2\in R$,
$x\in G$.
\end{itemize}
Indeed, just define ${}^x\!{r} = \alpha_x(r)$.

Given such an action, define the corresponding
semidirect product $R\rtimes G$
to be the group with underlying set $R\times G$ and
multiplication given by
\[(r,g)(r_1,g_1)= (r({}^g\!{r_1}), gg_1)\quad (r,r_1\in R, \ g,g_1\in G).\]
One checks that indeed this multiplication makes
$R\rtimes G$ into a group with identity element $(1,1)$. Note that
\[(r,g)^{-1}= ({}^{g^{-1}}\!{(r^{-1})} , g^{-1}),\quad
  (1,g)  (r,1)(1,g)^{-1} = ({}^g\!{r} ,1)\]
Moreover, the maps
\begin{equation*}
R\rightarrow R\rtimes G\quad r\mapsto (r,1)
\ \ (r\in R)\quad \text{and}\quad
G\rightarrow
R\rtimes G\quad g\mapsto (1,g)
\ \ (g\in G)
\end{equation*}
are injective homomorphisms. If we identify $R$ and
$G$ with their images under these injections, we have
$R\rtimes G=RG$, with $R\cap G=1$ and
$R\lhd R\rtimes G$. When using this identification   we sometimes write the elements of $R\rtimes G=RG$ as $r\cdot g$
($r\in R$, $g\in G$).  Throughout the paper we use the notation $(r, g)$ or $r\cdot g$ for an element of $R\rtimes G$, according to
convenience.
\subsection{Permutational Wreath Products}

Fix a set $\Sigma$.
Given a  group $A$, define  $A^\Sigma$
  to be the group of all functions $f\colon \Sigma\rightarrow
A$. We write the argument of such a function $f$ on
its right; thus  the operation on $A^\Sigma$ is
given by \[(f g)(s)= f(s)g(s) \quad (f,g\in A^\Sigma, s\in
\Sigma).\]
We denote by
$\delta\colon  A\rightarrow A^\Sigma$
 the {\it diagonal homomorphism}: it assigns to $a\in A$, the constant
function $\delta_a\in A^\Sigma$ defined by $\delta_a(s)= a$, for all $s\in \Sigma$. The image of $\delta$ is denoted
$\delta_A$.

Assume that a group $G$ acts on $\Sigma$ on the right. Define the {\it permutational wreath product} $A\wr G$  (with respect to the
$G$-set
$\Sigma$) to be the semidirect product
\[A\wr G= A^\Sigma \rtimes G,\]
where the action of $G$ on $A^\Sigma$ is defined by
\[{}^g\!{f}(s)= f(sg)\quad (g\in G , f\in A^\Sigma,
s\in \Sigma).\]
The usage of left exponentiation follows Eilenberg~\cite{Eilenberg}.
Observe that $G$  centralizes $\delta_A$ in $A\wr G$, so that $\langle \delta_A, G\rangle= \delta_A\times G$.


\subsubsection{Elementary Properties}
Several fundamental properties of the wreath product are recorded in the following proposition.

\begin{Prop}\label{elementarypropsofwreath}
{}\
\begin{itemize}
\item[(a)] If $B\le A$ are groups, and $H\le G$, then
\[B\wr H= B^\Sigma\rtimes H\le A\wr G=
A^\Sigma\rtimes G\]
\item[(b)] {\it Functoriality on $A$:}   $(-)\wr G$ is a functor, i.e.,
for each homomorphism $\alpha\colon  A\rightarrow B$,
there is a homomorphism
\[\alpha\wr G\colon  A\wr G= A^\Sigma\rtimes G\rightarrow B\wr G= B^\Sigma \rtimes G\]
given by $(f, g)\mapsto (\alpha f, g)$ \ $(f\in
A^\Sigma, g\in G)$ so that
\begin{itemize}
\item[(b1)] $\mathrm{id}_A\wr G=
\mathrm{id}_{A\wr G}$, and
\item[(b2)]  if $A\xrightarrow{\, \alpha\,} B\xrightarrow{\, \beta\,} C$ are group homomorphisms, then
$\beta\alpha \wr G= (\beta \wr G)(\alpha \wr G).$
\end{itemize}
\item [(c)] Furthermore, $\alpha\wr G$ is an epimorphism (respectively, monomorphism) if and only if $\alpha$ is an epimorphism (respectively, monomorphism).
\end{itemize}
\end{Prop}

\subsection{The Standard Embedding}
 Let $H$ be a subgroup of a group $G$. Let $\Sigma= H\backslash G$ be the
set of all right cosets of $H$ in $G$.
Denote by
$\rho\colon  G\rightarrow S_\Sigma$
the regular representation of $G$ in $S_\Sigma$, i.e., $\rho$ is the
homomorphism defined by $ \rho(g)=\bar g  $  $(g\in G)$, where $\bar g\colon  \Sigma \rightarrow
\Sigma$ is the permutation $Hx \mapsto  Hxg$ \ $(x\in G)$. Note that
\[\ker(\rho)= \bigcap _{x\in G}xH x^{-1} =H_G,\]
 the {\it core} of $H$ in $G$.

Fix a right transversal $T$ of $H$ in $G$, i.e., a complete set of representatives of the right cosets $Hx$
 $(x\in G)$. We denote the representative of $Hx$  in $T$ by either $t_{Hx}$ or $\ov x$, as convenient. Define
$s_T\in G^\Sigma$
as the map that assigns to each right coset of $H$ in $G$ its representative in $T$:
\[s_T (Hx)= t_{Hx}=
\ov x\in T \quad    (x\in G).\]
Consider the
monomorphism of groups $\widetilde{\varphi}\colon  G\rightarrow G\wr \rho(G)$
given by the composition of homomorphisms
\[\xymatrix{G\ar[r]^(.3){\delta  \times  \rho}\ &\delta_G\times  \rho(G)\ \ar@{^{(}->}[r] &G\wr \rho(G)\ar[rr]^{\mathrm{inn}_{S_T}}&&G\wr
\rho(G)}.\]
 Explicitly, if $g\in G$, then
\[\widetilde\varphi(g)= s_{T}(\delta_g\cdot\rho(g))s_{T}^{-1}= f_g\cdot \rho(g),\]
where $f_g\in G^\Sigma$ is defined by
 $f_g= s_{T}\delta_g{}^{\rho(g)}\!{(s_{T}^{-1})}$, i.e.,
\[f _g(Hx)= t_{Hx}g t_{Hxg}^{-1}  \quad (x\in G).\]

We remark  that $\widetilde\varphi(G)\le H\wr \rho(G)$, because $f_g(Hx)= t_{Hx}g t_{Hxg}^{-1}\in H$ ($x\in G$). Therefore, we have
proved
\begin{Thm}[Embedding Theorem]\label{embeddingtheorem}
Let $H\leq G$ be groups.
\begin{itemize}
\item[(a)] There is an injective homomorphism
$\varphi\colon  G\rightarrow H\wr \rho(G)$
defined by
\[\varphi(g)= f_g\cdot\rho(g)\]
where $f_g\colon \Sigma= H\backslash G\rightarrow H$ is given by $f_g(Hx)= t_{Hx}g t_{Hxg}^{-1}$  ($g, x \in G$).
\item[(b)] $\varphi|_{H}(H)\le H^\Sigma \rtimes \rho(H)=H\wr \rho(H)$.
\end{itemize}
\end{Thm}

We record the following facts for future use; they follow by routine
computation in the wreath product.

\begin{Lemma}\label{ribeslemma} Let $\psi\colon G\rightarrow K\wr \rho(G)$ be a
homomorphism such that
\[\xymatrix{G\ar[rr]^{\psi}\ar[rrd]_{\rho} &&K\wr \rho(G)\ar[d]^{\theta}\\ && \rho(G)}\]
commutes where $\theta$ is the projection. Put $\psi(g)= (\til f_g, \rho(g))$  $(g\in G)$. Then the following hold:
\begin{itemize}
\medskip
\item[(a)]  $\til f_{g_1g_2\cdots g_n}= \til f_{g_1}{}^{\rho(g_1)}\!{\til  f_{g_1}}\cdots{}^{\rho(g_1\cdots g_{n-1})}\!{\til  f_{g_n}}$ for
$g_1,\ldots,g_n\in G$;
\item[(b)] $\til f_{g^{-1}}= ({}^{\rho(g^{-1})}\!{\til f_g})^{-1}=
{}^{\rho(g^{-1})}\!{(\til f_g ^{-1})}$ for $g\in G$.
\end{itemize}
\end{Lemma}

\begin{Rmk}
 If $H\lhd G$, then $\Sigma$ has the structure of a group that we denote $K$. Identifying $K$ with its
canonical image in $S_\Sigma= S_K$, we have $K= \rho(G)$, so that $\varphi\colon  G\hookrightarrow H\wr K$. This is the so called Kaluznin-Krasner
Theorem: every extension of a group $H$ by a group $K$ can be embedded in $H\wr K$~\cite{KrasnerKal}.  The standard embedding is very closely related to the monomial map~\cite[Chapter 14]{Hallbook} and the theory of induced representations; see~\cite{wielandt} for a detailed discussion.
\end{Rmk}

From now on we shall use the notation  $T= \{t_i\mid i\in I\}$ (if $T$ is
finite, we write $T=\{t_1, \ldots , t_k\}$), and we shall
assume that there is a symbol $1\in I$ such that
$t_1=1$ is the representative of the
 coset $H$, i.e., $t_H=t_1=1$.  Fix $i\in I$.   Then the action of
$H^{t_i}= t_i^{-1}Ht_i$ on $\Sigma= H\backslash
G$ fixes the element $Ht_i\in \Sigma$.   Hence if $A$ is a group and
$f\in A^\Sigma$, one has ${}^{\rho(x)}\!{f} (Ht_i)=f(Ht_i)$,
for all $x\in H^{t_i}$. Therefore, the copy  \[\{f(Ht_i)\mid f\in
H^\Sigma\}\cong A\] of the group $A$ corresponding to the
$Ht_i\in \Sigma$ component of the direct product $A^\Sigma$ centralizes
$\rho(H^{t_i})$  in $A\wr \rho(H^{t_i})$. Thus \[A\wr
\rho(H^{t_i})= A^\Sigma \rtimes \rho(H^{t_i})= A\times
(A^{\Sigma-\{Ht_i\}}\rtimes \rho(H^{t_i})).\] We denote by
$\pi_{A,i}\colon  A\wr \rho(H^{t_i})\rightarrow A$ the corresponding
projection: \[\pi_{A,i}(f\cdot\rho(x))= f(Ht_i)\quad  (x\in
H^{t_i}, f\in A^\Sigma).\]  The case $i=1$ will be used so often, that it
is convenient to set $\pi_A=\pi_{A,1}$.  Part (b) of the
following lemma expresses the naturality of $\pi_{A,i}$.

\begin{Lemma}\label{naturality}
We continue with the  above setting. Let  $i\in I$.
\begin{itemize}
\item[(a)] There is a commutative diagram
\[\xymatrix{H^{t_i}\ar[rr]^{\p|_{H^{t_i}}}\ar[rrd]_{\mathrm{inn}_{t_i}|_{H^{t_i}}}
    &&   H\wr \rho(H^{t_i})\ar[d]^{\pi_{H,i}} \\ && H.}\]
In particular, for $i=1$, $\pi_H\varphi|_H = \pi_{H,1}\varphi|_{H}= {\rm
id}_H$.
\item[(b)] If $\alpha\colon  A\rightarrow B$ is a homomorphism of groups,
then the diagram
\[\xymatrix{A\wr \rho(H^{t_i})\ar[d]_{\pi_{A,i}}\ar[rr]^{\alpha \wr
\rho(H^{t_i})}&&B\wr \rho(H^{t_i})\ar[d]^{\pi_{B,i}}\\
A\ar[rr]_\alpha&&B}\]  commutes.
\item[(c)]  One has $\bigcap_{i\in I}(A\wr \rho(H^{t_i})) = A\wr \rho(H_G) =
A^{\Sigma}$, where $H_G$ is the core of $H$.
 The restriction $(\pi_{A,i})|_{A^{\Sigma}}\colon A^{\Sigma}\rightarrow A$
is the usual direct product projection.
\end{itemize}
\end{Lemma}
\begin{proof}
To prove (a)   observe that,  for $r\in H^{t_i}$, one has
\[\pi_{H,i}\varphi (r)=   f_r(Ht_i)= t_{Ht_i}rt_{Ht_ir}^{-1}=
t_irt_i^{-1}\] since $H^{t_i}$ stabilizes $Ht_i$.  The proof of (b)
follows directly from the definitions of $\pi_{A,i}$,
 $\pi_{B,i}$, and $\alpha \wr \rho(H^{t_i})$.  Part (c) is clear,  as
$\bigcap_{i\in I} H^{t_i}=H_G= \ker \rho$.
\end{proof}

\subsection{The Embedding Theorem for Profinite Groups}\label{profiniteversion}
By a \emph{variety} of finite groups we mean a nonempty class $\mathscr C$ of finite groups closed under
taking subgroups, finite direct products and homomorphic images. In this paper, we assume  in addition
that the variety
$\mathscr C$ is closed under extensions of groups  (we say then that $\mathscr C$ is an \emph{extension closed} variety of finite groups).  A pro-$\mathscr C$ group is  a profinite group whose continuous finite quotients are in $\mathscr C$, i.e., an inverse limit  of groups in $\mathscr C$.  Suppose now that $G$ is a pro-$\mathscr C$  group and $H$ is an open subgroup of $G$ (cf.~\cite{RZbook} for basic properties of profinite groups).  Let $\Sigma = H\backslash G$; then $\Sigma$ is finite, and the quotient topology on $\Sigma$ is discrete.
Let
$\rho\colon G\to S_{\Sigma}$ be as before; since
$H_G=\ker \rho$ is open in $G$, the homomorphism $\rho$ is
continuous.   If
$A$ is any pro-$\mathscr C$  group, then $A^{\Sigma}$ is a pro-$\mathscr C$  group and the left action $\rho(G)\times
A^{\Sigma}
 \to A^{\Sigma}$, as defined above, is continuous since  $\Sigma$ and $\rho(G)$ are finite and the action just permutes the coordinates.
  Thus the wreath product $A\wr \rho(G) = A^{\Sigma}\rtimes
\rho(G)$ is a pro-$\mathscr C$  group (here we use that $\mathscr C$ is extension closed).  Moreover, if $T$ is a
transversal for
$\Sigma$, then it follows immediately from the definition that the standard embedding $\varphi\colon
G\rightarrow H\wr \rho (G)$  is continuous.

\section{The Nielsen-Schreier Theorem}
We present an elementary proof of the Nielsen-Schreier Theorem, stating that subgroups of free groups are free, using wreath products. Our proof is algebraic in nature, rather than combinatorial, and proceeds by direct verification of the universal property.  Let $F$ be a free group on $X$ and $H$ a subgroup.  Elements of $F$ can be viewed as reduced words over $X\cup X\inv$~\cite{LyndonandSchupp}.

\subsection{Schreier Transversals}
A \emph{Schreier transversal} for $H\leq F$ is a right transversal $T$ of $H$ in $F$ that is closed under taking prefixes (and in particular contains the empty word): if $y_1\cdots y_i\cdots y_n\in T$ with $y_1,\ldots,y_n\in X\cup X\inv$ in reduced form, then $y_1\cdots y_i\in T$, for all $i=0,\ldots,n-1$.  The existence of Schreier transversals is a standard exercise in Zorn's Lemma.

\begin{Lemma}
There exists a Schreier transversal $T$ of $H$ in $F$.
\end{Lemma}
\begin{proof}
Consider the collection $\mathscr P$ of all prefix-closed sets of reduced words in $X\cup X^{-1}$ that intersect each right coset of $H$ in at most one element and order $\mathscr P$ by inclusion.  Then $\{1\}\in \mathscr P$, so it is non-empty.  It is also clear that the union of a chain of elements from $\mathscr P$ is again in $\mathscr P$, so $\mathscr P$ has a maximal element $T$ by Zorn's Lemma.  We need to show that each right coset of $H$ has a representative in $T$.
Suppose this is not the case and choose a minimum length word $w$ so that $Hw\cap T=\emptyset$.  Since $1\in T$, it follows $w\neq 1$ and hence $w=ux$ in reduced form where $x\in X\cup X\inv$.  By assumption on $w$, we have $Hu=Ht$ some $t\in T$.  If $tx$ is reduced as written, then $T\uplus \{tx\}\in \mathscr P$, contradicting the maximality of $T$.  If $tx$ is not reduced as written, then $tx\in T$ by closure of $T$ under prefixes and $Hw = Htx$, contradicting the choice of $w$.  This completes the proof that $T$ is a transversal.
\end{proof}

\subsection{The Nielsen-Schreier Theorem} We now proceed with our proof that subgroups of free groups are free via wreath products.

\begin{Thm}[Nielsen-Schreier]
Subgroups of free groups are free.  More precisely, let $F$ be a free group on $X$ and let $H$ be a subgroup. Let $T$ be a Schreier transversal for $H$ and
\begin{equation}\label{Schreierbasis}
B=  \{tx(\ov {tx})\inv \mid (t,x)\in T\times X, tx(\ov{tx})\inv\neq 1\}.
\end{equation}
 Then $H$ is freely generated by $B$.
\end{Thm}
\begin{proof}
Our goal is to show that any map $\alpha\colon B\to G$ with $G$ a group
extends uniquely to a homomorphism $\gamma\colon H\to G$.
First define an
extension \mbox{$\alpha\colon B\cup \{1\}\to G$} by $\alpha(1)
=1$.
Denote by $\Sigma$ the set
$H\backslash F$ of right cosets of $H$ in $F$ and let
$\rho\colon F\to S_{\Sigma}$ be the associated permutation
representation of $F$.

 To motivate our construction of the extension, we start with a proof of uniqueness.
So let $\gamma\colon H\to G$ be any homomorphism
extending $\alpha$.  Consider  the standard wreath product embedding
$\p\colon F\to  H\wr \rho(F)$ of Theorem~\ref{embeddingtheorem}. The
functoriality of the wreath product and Lemma~\ref{naturality} yield the commutative diagram
\[\xymatrix{F \ar[rr]^{\p} && H\wr \rho(F)  \ar[rr]^{\gamma\wr \rho(F)}& & G\wr \rho(F)\\
 H\ar@{^{(}->}[u]\ar[rr]^{\p|_{H}}\ar[drr]_{\mathrm{id}_H}        && H\wr \rho(H)\ar@{^{(}->}[u]\ar[rr]^{\gamma\wr \rho(H)}\ar[d] ^{\pi_H}&&G\wr \rho(H)\ar@{^{(}->}[u]\ar[d]^{\pi_G} \\
&& H \ar[rr]_{\gamma} && G.}
\]
Hence $\gamma$ is uniquely determined by $(\gamma\wr\rho(F))\p$, which is in turn determined by its values on $X$.  But if $x\in X$, then $(\gamma\wr \rho(F))\p(x) = (\gamma f_x,\rho(x))$. Now recall that  $f_x(Hw) = t_{Hw}xt_{Hwx}^{-1}\in B\cup \{1\}$ and hence $\gamma f_x = \alpha f_x$.  Thus the unique possible extension of $\alpha$ to a homomorphism is given by $\pi_G(\tau|_{H})$ where $\tau\colon F\to G\wr \rho(F)$ is the homomorphism defined on $X$ by $\tau(x) = (\alpha f_x,\rho(x))$.  Let us show $\pi_G(\tau|_{H})$ extends $\alpha$.

Let $b\in B$.  Then $b=tx(\ov {tx})\inv$ some $t\in T$, $x\in X$.  Let us suppose that $t=x_1\cdots x_{k-1}$ and $(\ov{tx})\inv = x_{k+1}\cdots x_n$ in reduced form.  We put $x_k=x$ so that $b=x_1\cdots x_n$, although this product may not be reduced as written.  Set $t_i = \ov{x_1\cdots x_i}$, for $i=0,\ldots,n$.
Using that Schreier transversals are prefix-closed one easily deduces the formulas:
\begin{equation}\label{schreiertransversal}
\begin{split}
t_i& =x_1\cdots x_i \qquad\quad  i<k\\
t_i&=x_n\inv\cdots x_{i+1}\inv\ \quad i>k.
\end{split}
\end{equation}
Indeed, the first formula is clear.  The second follows because, for $i\geq k+1$, $Ht_i = Htxx_{k+1}\cdots x_i=Hx_n\inv\cdots x_{k+1}\inv x_{k+1}\cdots x_i = Hx_n\inv \cdots x_{i+1}\inv$.

Our aim now is to verify $\pi_G\tau(b)=\alpha(b)$.  Put $\tau(r)=(f'_r,\rho(r))$, for $r\in F$.  We claim that if $t_{Hw}xt_{Hwx}\inv=1$ ($x\in X\cup X\inv, w\in F$), then $f'_x(Hw)=1$.  This is immediate if $x\in X$, since $f'_x=\alpha f_x$ and $f_x(Hw)=t_{Hw}xt_{Hwx}\inv$.  Next assume $x\in X\inv$. Hence, taking into account that  $x^{-1}\in X$, \[f'_x(Hw)= (\alpha f_{x^{-1}}(Hwx))^{-1}=
(\alpha(t_{Hwx}x^{-1}t_{Hw}^{-1}))^{-1}=1,\] since $t_{Hwx}x^{-1}t_{Hw}^{-1}=
(t_{Hw}xt_{Hwx}^{-1})^{-1}=1$.

In light of \eqref{schreiertransversal}  it follows that $t_{i-1}x_it_i\inv =1$ for all $i\neq k$.
  Thus by the claim and Lemma~\ref{ribeslemma}
\begin{align*}
\pi_G\tau(b) &=f'_b(H)= f'_{x_1\cdots x_n}(H)\\&= f'_{x_1}(H)f'_{x_2}(Hx_1)\cdots f'_{x_n}(Hx_1\cdots x_{n-1})\\
&= f'_{x_1}(Ht_0)f'_{x_2}(Ht_1)\cdots f'_{x_n}(Ht_{n-1})\\
       &= f'_{x_k}(Ht_{k-1}) = \alpha(t_{k-1}x_kt_k\inv)
        = \alpha(b)
\end{align*}
as required.  This completes the
proof that $H$ is freely generated by $B$.
\end{proof}

\begin{Rmk}
Notice that the above proof only shows that $B$ is a basis for $H$.  It does not follow from the proof that $B$ is in bijection with the set of pairs $(t,x)\in T\times X$ so that $tx(\ov{tx})\inv\neq 1$, although this can be deduced by straightforward combinatorial reasoning.
\end{Rmk}


\subsection{The Nielsen-Schreier Theorem for Free Profinite Groups}

Let $X$ be a profinite  space (i.e., a
compact Hausdorff and totally  disconnected topological space).  Then a pro-$\mathscr C$ group $F$ is said
to be a \emph{free pro-$\mathscr C$ group} on
$X$ if there is a continuous map
$\iota\colon X\rightarrow F$ so that if $\sigma\colon X\rightarrow G$ is any continuous map into a
pro-$\mathscr C$ group $G$, then there is a unique continuous homomorphism $\psi\colon F\rightarrow G$ so
that \[\xymatrix{X\ar[r]^{\iota}\ar[rd]_{\sigma} &F\ar@{.>}[d]^{\psi}\\ & G}\]
 commutes.  If $(X,\ast)$ is a pointed profinite space, one defines in an  analogous manner the concept
of free pro-$\mathscr C$ group on $(X,\ast)$:  it satisfies the same universal property as above, but with
all the maps assumed to be  continuous maps of pointed spaces (the maps send distinguished points to
distinguished points; the distinguished point of a group being its identity element). The map $\iota$
is an embedding, and we identify $X$ with its image $\iota(X)$.

  Observe that a free pro-$\mathscr C$ group $F$ on a profinite space $X$ can be viewed as
free pro-$\mathscr C$ group on the pointed profinite space $(X\uplus \{\ast\}, \ast)$ in an obvious way; so
we deal here only with pointed spaces. Let $\Phi$ denote the abstract subgroup of $F$ generated by $X$.
Then (cf.~\cite[Propositions 3.3.13 and 3.3.15]{RZbook}), $\Phi$ is a free abstract group with basis $X-
\{\ast\}$; furthermore $\Phi$ is dense in $F$.

  Let
$H$ be an open subgroup of $F$. Then  the natural map \[(H \cap \Phi)\backslash \Phi \rightarrow
 H\backslash F =\Sigma\] is a bijection.  Choose a  Schreier transversal $T$ of $H\cap\Phi$ in $\Phi$.
The map \[T\times X\rightarrow B= \{tx(\overline {tx})^{-1}\mid t\in T, x\in X\} \subseteq H\le F\]
 given by
$(t,x)\mapsto
 tx(\overline {tx})^{-1}=
 tx(t_{\pi(tx)})^{-1}$  (where $\pi\colon F\rightarrow \Sigma =H\backslash F$
  is the projection) is continuous,
 since $\pi$ and the section  $Hf\mapsto t_{Hf}$ from $H\backslash F$ to
  $F$ are obviously continuous.
 And  so,
  $B$ is closed by the compactness of
  $T\times X$,  i.e., $B$
   is profinite.  Observe that $1\in  B$. We think of $B$ as a pointed space with distinguished point
$1$.   The proof of the
Nielsen-Schreier Theorem that we have presented above now goes through  {\it mutatis mutandi} to show
that
$H$ is a free profinite group on the pointed space $(B, 1)$.  Thus we have:

\begin{Thm} Let $\mathscr C$ be an extension closed variety of finite groups. Let  $F$ be
a free pro-$\mathscr C$ group on a pointed profinite space
$(X, \ast)$ and let
$H$ be an open subgroup of
$F$.  Then $H$ is a free pro-$\mathscr C$ group on a pointed profinite space.
\end{Thm}

\section{The Kurosh Theorem}
In this section, we give what may arguably be considered the first algebraic proof of the Kurosh Theorem on subgroups of free products.  The original proof is essentially combinatorial, while modern proofs have a topological character.  Perhaps, Higgin's proof can also be considered algebraic, but it relies on groupoids~\cite{Higgins}.  Our proof has a similar flavor to the above proof of the Nielsen-Schreier Theorem in that it relies on wreath products.  A key difference is that the transversals used are more complicated.

\subsection{Free Products}
Let $G=\bigast_{\alpha\in A} G_{\alpha}$ be the free product of the groups $G_{\alpha}$, $\alpha\in A$.  We shall freely use the Normal Form Theorem for free products~\cite[Chapter IV]{LyndonandSchupp}, stating that each non-trivial element $g$ of $G$ can be uniquely written in the form $g=g_1g_2\cdots g_m$ where each $g_i$ belongs to some $G_{\alpha}$ and $g_i\in G_{\alpha}$ implies $g_{i+1}\notin G_{\alpha}$, for $i=1,\ldots, m-1$.    The number $m$ will be called the \emph{syllable length} of $g$ and we write $\ell(g)=m$. If $S\subseteq G$, denote by $\ell(S)$ the smallest syllable length of an element of $S$.  By convention, the syllable length of the identity is $0$.  If $g_m\in G_{\alpha}$, then we shall say that $g$ ends in the syllable $\alpha$ or that $\alpha$ is the last syllable of $g$.

\subsection{Kurosh Systems}
Let us begin by setting up notation.  Suppose $G =\bigast_{\alpha\in A}
G_{\alpha}$.   Let $H\leq G$ and set
$\Sigma = H\backslash G$.  Denote by $\rho\colon G\to S_{\Sigma}$ the
action map. Let $\{H_i\mid i\in I\}$, be the right cosets of $H$ and assume
there is a symbol $1\in I$ such that $H_1=H$.  Assume that we have a
transversal
$T_{\alpha}$ of the right cosets of $H$ in $G$  for each $\alpha\in
A$. Denote by $\alpha(H_i)$ the representative of $H_i$ in
$T_{\alpha}$. We require $\alpha(H) = 1$, all $\alpha\in A$.

\begin{Def}[Kurosh System]
A collection $D=\{D_{\alpha}\mid \alpha\in A\}$ of systems
$D_{\alpha}$ of  representatives
$\alpha(HgG_{\alpha})$ of the double
cosets $H\backslash G/G_{\alpha}$, \mbox{$\alpha\in A$}, together with a
system $\{T_{\alpha}\mid \alpha\in A\}$ of transversals for
$H\backslash G$ is called a \emph{Kurosh system} if the following holds:
\begin{itemize}
\item [(i)] If $g = \alpha(HgG_{\alpha})$, then $g = \alpha(Hg)$;
\item [(ii)] $\alpha(HgG_{\alpha})$ is either $1$ or ends in a
syllable $\beta\neq\alpha$;
\item [(iii)] $H_i\subseteq HgG_{\alpha}$ and
$\alpha(HgG_{\alpha})=g$ implies $\alpha(H_i) \in gG_{\alpha}$;
\item [(iv)] If $1\neq g=\alpha(HgG_{\alpha})$ has last syllable in
$G_{\beta}$, then $\beta(Hg)=g$;
\item [(v)] $\ell(\alpha(HgG_\alpha))= \ell
(HgG_\alpha)$.
\end{itemize}
\end{Def}

\begin{Prop}
Kurosh systems exist.
\end{Prop}
\begin{proof}
We proceed by induction on the length of the double cosets
$HgG_\alpha$. If $\ell(HgG_\alpha)=0$, i.e.,
 $HgG_\alpha= HG_\alpha$ choose
$\alpha(HgG_\alpha)=1$ and   $\alpha(H)=1$;  if $H\ne H_i\subseteq
HG_\alpha$, choose $a_\alpha\in G_\alpha$ so that $H_i=Ha_\alpha$, and put
$\alpha(H_i)= a_\alpha$. Then conditions (i)--(v) hold. Let $n>1$,
and assume representatives $\beta(HrG_\beta)$ and $\beta(H_i)$ have been
chosen whenever $H_i\subseteq HrG_\beta$ and $\ell( HrG_\beta)\le
n-1$ $(\beta\in A, r\in G)$, satisfying conditions
(i)--(v). Let
$\ell (HgG_\alpha)= n$ with $\ell(g)=n$. Then $g= \bar g a_\beta$, where $\ell(\bar
g)=n-1$, $1\ne a_\beta\in G_\beta$ and $\beta\ne \alpha$. Since $\ell(HgG_\beta)\leq n-1$, representatives
$\beta(HgG_\beta)=t$ and $\beta (Hg)=tb_\beta$  ($b_\beta\in G_\beta)$ have already been chosen; in particular, $\ell(t)\leq n-1$ by (v). Since
$\ell (Hg)=n$, we deduce that $b_\beta\ne 1$ and $\ell(tb_\beta)=n$. Define $\alpha(HgG_\alpha)=tb_\beta= \alpha(Hg)$, and
whenever
$Hg\ne H_i\subseteq HgG_\alpha$, choose $c_\alpha\in G_\alpha$ so that $H_i=Hgc_\alpha$, and put
$\alpha(H_i)= tb_\beta c_\alpha$. Clearly, conditions (i)--(v) are satisfied.
\end{proof}

Let us define some key elements of $H$.   Fix an index $\alpha_0\in
A$. For $x\in G_{\alpha}$ and $H_i\in
H\backslash G$, define:
\begin{align*}
y_{i,x} &= \alpha(H_i)x\alpha(H_ix)\inv;\\
z_{i,\alpha} &= \alpha(H_i)\alpha_0(H_i)\inv.
\end{align*}
 It is immediate that $y_{i,x}, z_{i,\alpha}\in H$ all $i$, $x$ and
$\alpha$.  Notice that $z_{1,\alpha}=1=z_{i,\alpha_0}$ for all
$\alpha\in A$, $i\in I$.  If $H_i= Hg$, we often write $y_{Hg,x}$
and $z_{Hg,\alpha}$ for $y_{i,x}$ and $z_{i,\alpha}$.
 We begin with some simple observations concerning these
elements.

\begin{Prop}\label{relatortypes}
Retaining the above notation, we have:
\begin{enumerate}
\item If $x_1,x_2\in G_{\alpha}$, then $y_{i,x_1}y_{j,x_2} =
y_{i,x_1x_2}$ where $H_ix_1=H_j$;
\item If $x\in G_{\alpha}$, $H_i\subseteq HuG_{\alpha}$ with
$u=\alpha(HuG_{\alpha})$, then $y_{i,x}\in uG_{\alpha}u\inv\cap H$;
\item If $h\in uG_{\alpha}u\inv\cap H$ with $u=\alpha(HuG_{\alpha})$, then
$h=y_{Hu,x}$ for some $x\in G_{\alpha}$;
\item If $1\neq u=\alpha(HuG_{\alpha})$ ends with a
$\beta$-syllable, then $z_{Hu,\alpha}=z_{Hu,\beta}$.
\end{enumerate}
\end{Prop}
\begin{proof}
First we handle (1).  Straightforward computation yields
\[y_{i,x_1}y_{j,x_2}
=\alpha(H_i)x_1\alpha(H_ix_1)\inv\alpha(H_ix_1)x_2\alpha(H_ix_1x_2)\inv
= y_{i,x_1x_2}.\] Next we turn to (2).
 By condition (iii) of a Kurosh system, $\alpha(H_i)=ug$ and $\alpha(H_ix)=ug'$ some $g,g'\in
G_{\alpha}$, whence \mbox{$y_{i,x} = ugx(ug')\inv \in
uG_{\alpha}u\inv \cap H$}. To prove (3), suppose $h = uxu\inv$ with
$x\in G_{\alpha}$. Then $Hu=Hux$ and $\alpha(Hu)=u$ by (i). We
conclude $y_{Hu,x} = \alpha(Hu)x\alpha(Hux)\inv = uxu\inv =h$. For
(4) we simply observe $\alpha(Hu) = u=\beta(Hu)$ by (i) and (iv).
\end{proof}

Set  $Z= \{z_{i,\alpha}\mid i\in I, \alpha\in A, z_{i,\alpha}\neq 1\}$ and
$F=\langle Z\rangle$.

\subsection{The Kurosh Theorem}
Our goal is to prove \[H = \bigast_{\alpha\in A}\left[\bigast_{u\in D_{\alpha}}
\left(uG_{\alpha}u\inv\cap H\right)\right] \mathrel{\ast} F\] and $F$ is freely
generated by $Z$. We use wreath products and the universal property
to effect this proof.  From now on we work with a fixed Kurosh
system.  If $\psi\colon Z\to K$ is a map, with $K$ a group, then we extend
$\psi$ to $Z\cup \{1\}$ by setting $\psi(1) =1$.

\begin{Prop}\label{existenceprop}
Given a family  $\mathscr F=\{\psi_u\colon uG_{\alpha}u\inv\cap H\to
K\}_{\alpha\in A, u\in D_{\alpha}}$ of
group homomorphisms and a map $\psi\colon Z\to K$,  there exists, for
each $\alpha\in A$, a homomorphism $\Psi_{\alpha}\colon G_{\alpha}\to
K\wr \rho(G)$ defined by $\Psi_{\alpha}(x)
= (f_x,\rho(x))$ with
\begin{equation*}
f_x(H_i)= \psi(z_{i,\alpha})\inv \psi_u(y_{i,x})\psi(z_{j,\alpha})
\end{equation*}
where $H_ix=H_j$ and $u=\alpha(H_iG_{\alpha})$.  If
\mbox{$\Psi\colon G\to K\wr \rho(G)$} denotes the induced
homomorphism, then the following diagram commutes
\begin{equation}\label{firstcommute}
\xymatrix{G\ar[rr]^{\Psi}\ar[rrd]_{\rho} && K\wr \rho(G)\ar[d]^{\theta}\\
                  & & \rho(G)}
\end{equation}
where $\theta$ is the projection.

Moreover, the construction of $\Psi$ is functorial in the sense that
given another family of homomorphisms $\mathscr F'=\{\psi_u'\colon uG_{\alpha}u\inv\cap
H\to K'\}_{\alpha\in A, u\in D_{\alpha}}$, a map \mbox{$\psi'\colon Z\to K'$} and a homomorphism
$\gamma\colon K'\to K$ such that the diagrams
\begin{equation}\label{firstcommuteb}
\xymatrix{&\ar[ld]_{\psi_{u}'} uG_{\alpha}u\inv \cap H\ar[rd]^{\psi_{u}} & &\\
                  K'\ar[rr]_{\gamma} & &K}\quad \xymatrix{& \ar[ld]_{\psi'} Z\ar[rd]^{\psi} & &\\
                  K'\ar[rr]_{\gamma} & &K}
\end{equation}
commute,  then the following diagram commutes
\begin{equation}\label{secondcommute}
\xymatrix{& &K'\wr \rho(G)\ar[d]^{\gamma\wr \rho(G)} \\
 G \ar[rru]^{\Psi'} \ar[rr]_{\Psi} && K\wr \rho(G)}
\end{equation}
where $\Psi'$ is the map associated to the family $\mathscr F'$.
\end{Prop}
\begin{proof}
We begin by verify that
$\Psi_{\alpha}$ is a homomorphism.  Proposition~\ref{relatortypes}(2) implies
$y_{i,x}\in uG_{\alpha}u\inv\cap H$ so that $\Psi_{\alpha}$ makes
sense. Let $x_1,x_2\in G_{\alpha}$. Clearly,
$H_ix_1G_{\alpha}=H_ix_2G_{\alpha} = H_ix_1x_2G_{\alpha}=H_iG_{\alpha}$; set
$u=\alpha(H_iG_{\alpha})$.  From
\[(f_{x_1},\rho(x_1))(f_{x_2},\rho(x_2)) =
(f_{x_1}({}^{\rho(x_1)}\!{f_{x_2}}),\rho(x_1x_2))\] it follows that we just need
$f_{x_1}(H_i)f_{x_2}(H_ix_1) = f_{x_1x_2}(H_i)$.  Putting $H_j=H_ix_1$ and
$H_k=H_ix_1x_2$, an application of Proposition~\ref{relatortypes}(1)
yields
\begin{align*}
f_{x_1}(H_i)f_{x_2}(H_ix_1)&= \psi(z_{i,\alpha})\inv
\psi_u(y_{i,x_1})\psi(z_{j,\alpha})  \psi(z_{j,\alpha})\inv
\psi_u(y_{j,x_2})\psi(z_{k,\alpha})\\
&= \psi(z_{i,\alpha})\inv \psi_u(y_{i,x_1x_2})\psi(z_{k,\alpha}) =
f_{x_1x_2}(H_i),
\end{align*}
as required.  The $\Psi_{\alpha}$ induce the desired map $\Psi$ by
the universal property of a free product.
 The commutativity
of \eqref{firstcommute} and \eqref{secondcommute} are immediate from
the definition of $\Psi_{\alpha}$ and the universal property of a
free product.
\end{proof}

As a corollary of the proposition and Lemma~\ref{naturality}, we obtain
\begin{Cor}\label{finalcommute}
Let $\Psi$, $\Psi'$ and $\gamma$ be as in
Proposition~\ref{existenceprop}.  Then there is a commutative
diagram
\begin{equation}\label{thirdcommute}
\xymatrix{ && K'\wr \rho(H)  \ar[rr]^{\pi_{K'}}\ar[d]^{\gamma\wr \rho(H)} & &K'\ar[d]^{\gamma}\\
                  H\ar[rru] ^{\Psi'|_H}   \ar[rr]_{\Psi|_H}& & K\wr \rho(H)\ar[rr]_{\pi_K}& & K}
\end{equation}
\end{Cor}

Our next lemma is where we make use of the full strength of the
Kurosh system.   It is convenient

\begin{Lemma}\label{effectontransversal}
Let $u = \alpha(Hu)$.  Then $f_u(H)=\psi(z_{Hu,\alpha})$.
\end{Lemma}
\begin{proof}
We induct on the syllable length of $u$.  If $u=1$, there is nothing
to prove as $z_{H,\alpha}=1$ all $\alpha$.  So assume $u\neq 1$.
The proof divides into two cases.
\begin{Case} Assume $u\neq \alpha(HuG_{\alpha})$.  Then
 (iii) implies we can write $u=vx$
with \mbox{$v=\alpha(HuG_{\alpha})$} and $x\in G_{\alpha}$.  Moreover, $\ell(v)<\ell(u)$ by (ii).  Since $\alpha(Hv)=v$ by (i),
by induction $f_v(H) = \psi(z_{Hv,\alpha})$.  Then we find by Lemma~\ref{ribeslemma}
\begin{align*}
f_u(H) = f_v(H)f_x(Hv) &= \psi(z_{Hv,\alpha})
\psi(z_{Hv,\alpha})\inv\psi_v(y_{Hv,x})\psi(z_{Hu,\alpha})\\
&= \psi_v(y_{Hv,x})\psi(z_{Hu,\alpha}).
\end{align*}
But $y_{Hv,x} = \alpha(Hv)x\alpha(Hvx)\inv =
\alpha(Hv)x\alpha(Hu)\inv = vxu\inv =1$, establishing $f_u(H) =
\psi(z_{Hu,\alpha})$.
\end{Case}
\begin{Case} Suppose $u=\alpha(HuG_{\alpha})$.  Since
$u\neq 1$, (ii) implies $u$ ends in a syllable $\beta$ with
$\beta\neq \alpha$ and (iv) yields $\beta(Hu)=u$.  By (ii) $u\neq
\beta(HuG_{\beta})$, so Case 1 implies $f_u(H) = \psi(z_{Hu,\beta})$. Proposition 4.2(4) provides $z_{Hu,\beta}= z_{Hu,\alpha}$, so $f_u(H)=
\psi(z_{Hu,\alpha})$.
\end{Case}
This establishes the lemma.
\end{proof}

An important special case is when $K=H$ and the $\psi_u$ and $\psi$ are
the inclusions. Let us denote the induced map in this case by
$\til \Psi\colon G\to H\wr \rho(G)$.
\begin{Prop}\label{realembedding}
The map $\til \Psi\colon G\to H\wr \rho(G)$ is the standard wreath product embedding
associated to the transversal $T_{\alpha_0}$. Consequently,
$\pi_H\til \Psi|_{H}$ is the identity.
\end{Prop}
\begin{proof}
Writing $\til \Psi(g) = (F_g,\rho(g))$, if $x\in G_{\alpha}$ and
$H_ix=H_j$, then
\begin{align*}
F_x(H_i)  = z_{i,\alpha}\inv y_{i,x}z_{j,\alpha} &=
\alpha_0(H_i)\alpha(H_i)\inv [\alpha(H_i)x\alpha(H_j)\inv]
\alpha(H_j)\alpha_0(H_j)\inv\\ &= \alpha_0(H_i)x\alpha_0(H_ix)\inv.
\end{align*}
Thus $\til \Psi$ is the standard embedding associated to the transversal
$T_{\alpha_0}$.
\end{proof}

In the proof of the next theorem, we retain all the notation introduced in this section.

\begin{Thm}[Kurosh]\label{kurosh}Let $\{D_\alpha, T_\alpha\mid \alpha\in A\}$ be a Kurosh system for $H\leq G=
\bigast _{\alpha\in A} G_{\alpha}$.  Then \[H = \bigast_{\alpha\in
  A}\left[\bigast_{u\in D_{\alpha}}
\left(uG_{\alpha}u\inv\cap H\right
)\right] \mathrel{\ast} F\] and $F$ is a free group
with basis $Z$.
\end{Thm}
\begin{proof}
Let $\{\psi_u\colon uG_{\alpha}u\inv\cap H\to K\}_{\alpha\in A, u\in D_{\alpha}}$ be a family of
group homomorphisms and $\psi\colon Z\to K$ a map.  Let $\Psi\colon G\to K\wr
\rho(G)$ be as in Proposition~\ref{existenceprop}.   We show
$\pi_K\Psi|_{H}$ extends the $\psi_u$ and $\psi$ where $\pi_K= \pi_{K,1}$ is as in Lemma~\ref{naturality}. Suppose
$u=\alpha(HuG_{\alpha})$ and $h\in uG_{\alpha}u\inv\cap H$. By
Proposition~\ref{relatortypes}(3), $h= y_{i,x}$ for some $x\in
G_{\alpha}$, where $H_i=Hu$. Setting $H_j=H_ix$, an application of
Lemmas~\ref{effectontransversal} and~\ref{ribeslemma} (and the fact $H\alpha(H_i)=H_i$)
yields
\begin{align*}\pi_K\Psi(h)&= f_{y_{i,x}}(H)= f_{\alpha(H_i)x\alpha(H_j)^{-1}}(H)\\&
 = f_{\alpha(H_i)}(H)f_x
(H_i)\big(f_{\alpha(H_j)}(H_j\alpha(H_j)^{-1})\big)^{-1}\\
 &=f_{\alpha(H_i)}(H)f_x(H_i)(f_{\alpha(H_j)}(H))^{-1} \\
& =\psi(z_{i,\alpha})[\psi(z_{i,\alpha})^{-1} \psi_u(y_{i,x})
\psi(z_{j,\alpha})]\psi(z_{j,\alpha})^{-1} \\
 &= \psi_u(y_{i,x})= \psi_u(h).
 \end{align*}

Similarly, we calculate using Lemmas~\ref{effectontransversal} and~\ref{ribeslemma}
\begin{align*}
\pi_K\Psi(z_{i,\alpha})&=f_{z_{i,\alpha}}(H) =
f_{\alpha(H_i)}(H)\big(f_{\alpha_0(H_i)}(H_i\alpha_0(H_i)\inv)\big)\inv
\\ &= f_{\alpha(H_i)}(H)(f_{\alpha_0(H_i)}(H))\inv \\ &= \psi(z_{i,\alpha})
\psi(z_{i,\alpha_0})\inv   = \psi(z_{i,\alpha})
\end{align*}
since $z_{i,\alpha_0}=1$.

The uniqueness of $\pi_K\Psi|_{H}$ follows from the
functoriality of our construction. Namely, in
Proposition~\ref{existenceprop} take $K'=H$ and $\psi'_u$, $\psi'$ the
inclusions (and so $\Psi'\colon G\to H\wr \rho(G)$ is $\til \Psi$ from Proposition~\ref{realembedding}). Suppose
$\gamma\colon H\to K$ is an extension of the $\psi_u$ and $\psi$. Then
\eqref{firstcommuteb} commutes and so diagrams \eqref{secondcommute}
and \eqref{thirdcommute} commute. Since $\pi_H\Psi'|_{H}=\pi_H\til{\Psi}|_H$
is the identity in this case by Proposition~\ref{realembedding}, we conclude
$\gamma=\pi_K\Psi|_{H}$.
\end{proof}

\begin{Rmk}
As we mentioned earlier, there is a close relationship between the standard embedding and induced representations~\cite{wielandt}.  From this viewpoint, our proof of the Kurosh Theorem has a similar flavor to Mackey's Theorem on the restriction to one subgroup of a representation induced from another.
\end{Rmk}

\subsection{The Kurosh Subgroup Theorem for Profinite Groups}
Let $\Gamma$ be a pro-$\mathscr C$ group and let
$\{\Gamma_\alpha \mid \alpha \in A\}$ be a collection of pro-$\mathscr C$ groups indexed by a set $A$. For
each $\alpha\in A$, let $\iota_\alpha\colon \Gamma_\alpha\rightarrow \Gamma$ be a
 continuous homomorphism. One says that the family $\{\iota_\alpha\mid
\alpha\in A\}$ is \emph{convergent} if whenever $U$  is an open
neighborhood of $1$ in $\Gamma$, then $U$ contains all but a finite number of
the images $\iota_\alpha(\Gamma_\alpha)$. We say that $\Gamma$ together with the $\iota_\alpha$
is the \emph{free pro-$\mathscr C$ product} of the groups $\Gamma_\alpha$ if the following universal
property is satisfied: whenever $K$ is a pro-$\mathscr C$ group and
$\{\lambda_\alpha\colon \Gamma_\alpha \rightarrow K\mid \alpha\in A\}$ is a
convergent family of continuous homomorphisms, then there exists a unique
continuous homomorphism $\lambda\colon \Gamma\rightarrow K$ such that
\[\xymatrix{\Gamma_{\alpha}\ar[r]^{\iota_\alpha}\ar[rd]_{\lambda_\alpha} &\Gamma\ar@{.>}[d]^{\lambda}\\ & K}\]
commutes, for all $\alpha\in A$. We denote such a free pro-$\mathscr C$
product by
$\Gamma=
\coprod_{\alpha\in A}\Gamma_\alpha $. Free pro-$\mathscr C$ products exist and are
unique. To construct the free pro-$\mathscr C$ product $\Gamma$ one proceeds as follows: let  $G=
\bigast_{\alpha\in A}\Gamma_\alpha$ be the free product of the groups $\Gamma_\alpha$ as abstract groups.
Consider the pro-$\mathscr C$ topology on $G$ determined by the collection of normal subgroups $N$ of
finite index in
$G$ such that $G/N\in \mathscr C$, $N\cap \Gamma_\alpha$ is open in $\Gamma_\alpha$, for each $\alpha\in
A$,  and $N\ge \Gamma_\alpha$, for all but finitely many   $\alpha$.  Then
\[\Gamma= \ilim_N G/N.\]
It turns out that $G$ is naturally embedded in $\Gamma$ as a dense subgroup. One can take the
homomorphism $\iota_\alpha$ to be the composition of inclusions \[\Gamma_\alpha\hookrightarrow
G\hookrightarrow \Gamma\quad  (\alpha\in A).\]  If the set $A$ is finite, the `convergence' property of
the homomorphisms $\iota_\alpha$ is automatic.

For such free products, one has the following analogue of
the Kurosh Subgroup Theorem~\cite{Wenzel}.

\begin{Thm} Let $H$ be an open subgroup of the free pro-$\mathscr C$ product
\[\Gamma= \coprod_{\alpha\in A}\Gamma_\alpha.\]
Then, for each $\alpha\in A$, there exists a set  $D_\alpha$ of representatives  of the double cosets
$H\backslash \Gamma/\Gamma_\alpha$ such that the family of inclusions \[\{u\Gamma_\alpha u^{-1}\cap
H\hookrightarrow H\mid u\in D_\alpha, \alpha\in A\}\] converges, and  $H$ is the free pro-$\mathscr C$
product
\[H= \left[\coprod_{\alpha\in A, u\in D_\alpha} u\Gamma_\alpha u^{-1}\cap H\right]\amalg \Phi,\]
where $\Phi$ is a free pro-$\mathscr C$ group of finite rank.
\end{Thm}
\begin{proof}
First we show that we may assume that $A$ is finite. Consider the core $H_\Gamma= \bigcap_{g\in
\Gamma}gH g^{-1}$ of $H$ in $\Gamma$. Since  $H$ is open, we have  that
$H_\Gamma$ is open in $\Gamma$. So there exists a finite subset $B$ of $A$ such that $A_\alpha\le
H_\Gamma$ for all $\alpha\in A-B$. Put $\Gamma' = \coprod_{\alpha\in A-B}\Gamma_\alpha $; then
\[\Gamma=  \left[\coprod _{\alpha\in  B}\Gamma_\alpha\right]\amalg \Gamma'\]
is a free pro-$\mathscr C$ product of finitely many factors, and one easily sees that it suffices to
prove the theorem for this product. Indeed,
observe first that for all $\alpha\in A-B$, $H_\Gamma\ge \Gamma_\alpha$  and since
$H_\Gamma\lhd
\Gamma$, one has $Hu\Gamma_\alpha= Hu= Hu\Gamma'$   $(u\in \Gamma)$, i.e., $H\backslash
\Gamma/\Gamma'=
 H\backslash \Gamma= H\backslash \Gamma/\Gamma_\alpha$; on the other hand,
 \[u\Gamma'u^{-1}\cap H= u\Gamma'u^{-1}= \coprod_{\alpha\in A-B}u\Gamma_\alpha u^{-1}=
\coprod_{\alpha\in A-B}(u\Gamma_\alpha u^{-1}\cap H).\]  Hence from now on we assume that $A$ is a finite indexing set.

Choose a Kurosh system $\{D_\alpha, T_\alpha \mid \alpha\in A\}$ for the subgroup
$G\cap H$ of the abstract free product $G= \bigast_{\alpha\in A}\Gamma_\alpha$, and observe that, for each
$\alpha$,
$T_\alpha$ and
$D_\alpha$  are systems of representatives of the cosets $H\backslash G$ and of the double
cosets
$H\backslash \Gamma /\Gamma_\alpha$, respectively. The remainder of the proof can be carried out \textit{mutatis mutandis} as is done in the proof of Theorem~\ref{kurosh}
  (one simply has to require initially that the
homomorphisms $\psi_u$  are continuous, and then verify that all the
maps involved in  the proof are also continuous; this is an easy consequence of our comments in \ref{profiniteversion}).
\end{proof}

Let us point out that this proof is independent of the result for abstract free products (Theorem~\ref{kurosh}); it
simply follows the same procedure.

 We leave open the question of whether or not the same simple procedure works in case  one deals
with   pro-$\mathscr C$ products of pro-$\mathscr C$ groups indexed by a profinite space~\cite{zalesskiifree}.

\section{Quasifree Profinite Groups}
This section contains a simpler proof of the main result of~\cite{quasifree}. A similar approach, using the twisted wreath product, was independently discovered by Bary-Soroker \textit{et al.}~\cite{lior}.

\subsection{Quasifree Groups}
An epimorphism of groups is termed \emph{proper} if it is not an
isomorphism. Let  $\mathscr C$ be an extension closed
variety of finite groups and let $m$ be  an infinite cardinal
number. A pro-${\mathscr C}$  group $G$ is called \emph{$m$-quasifree}
if whenever $A$  and  $B$  are  groups in ${\mathscr C}$, $\alpha
\colon    A\rightarrow B$ is a proper epimorphism of groups that
splits (i.e., there is a section $\sigma \colon  B\rightarrow A$ of
$\alpha$: $\alpha \sigma= {\rm id}_B$ ), and $\beta\colon G\rightarrow
B$ is a continuous  epimorphism,
\[\xymatrix{&&G\ar@{->>}[d]^\beta\ar@{..>>}[lld]_\lambda \\
  A\ar@{->>}[rr]^\alpha&&B\ar@/^/[ll]^\sigma}\] then there exist
precisely $m$ different continuous epimorphisms $\lambda\colon
G\rightarrow A$ such that  $\alpha \lambda= \beta$. See~\cite{quasifree,harbater} for
motivation and elementary properties of these groups; one knows in
particular that the minimal number $d(G)$ of generators converging to
$1$ of such an $m$-quasifree group $G$ is $m$ (see~\cite[Lemma 1.2]{quasifree}). In~\cite{quasifree} it is proved that open subgroups of $m$-quasifree groups are
$m$-quasifree. Here we provide a simpler and more natural proof of
this result by means of wreath products.

\begin{Thm}   Let $G$ be an $m$-quasifree pro-${\mathscr C}$ group, and let $H$ be an open subgroup of $G$. Then $H$ is $m$-quasifree.
\end{Thm}
\begin{proof}
Given  $A, B\in {\mathscr C}$, a proper  split epimorphism $\alpha
\colon A \rightarrow B$ and a continuous epimorphism $\beta\colon
H\rightarrow B$, we need to  prove the existence of exactly $m$
continuous  epimorphisms $\lambda\colon  H\rightarrow A$ such that
$\alpha \lambda=\beta$.

   Set $\Sigma= H\backslash G$ and let
$\rho\colon  G\rightarrow S_\Sigma$ be the corresponding permutation
representation as in Section~\ref{section1}.  Consider the  standard
embedding \[\varphi\colon  G\rightarrow H\wr \rho (G)\] constructed in
Theorem~\ref{embeddingtheorem}. Note that $\alpha\wr \rho (G)\colon
A\wr \rho (G)\rightarrow B\wr \rho (G)$ is   a split proper epimorphism by
Proposition~\ref{elementarypropsofwreath}; observe also that  $A\wr
\rho (G)$ and $B\wr \rho (G)$ are finite groups in ${\mathscr C}$, as
${\mathscr C}$ is extension closed.    Let $B' = (\beta\wr
\rho(G))\varphi(G)$ and $A' = (\alpha\wr \rho(G))^{-1}(B')$.  Then
$A',B'\in {\mathscr C}$,  and the restriction $\alpha'\colon
A'\rightarrow B'$ of $\alpha\wr \rho (G)$ to $A'$ is a split proper
epimorphism.  See Figure~\ref{unnamed}.
\begin{figure}[h!b!t!p!]
 \[\xymatrix{ A\wr \rho (G)\ar@{->>}[rr]^{\alpha\wr \rho
    (G)}&& B\wr \rho (G)\\
  A'\ar@{^{(}->}[u]\ar[rr]^{\alpha'}&&B'\ar@{^{(}->}[u]\\
  &&H\wr \rho(G)\ar[u]_{\beta\wr \rho (G)}\\
  &&G\ar[u]_\varphi\ar@{.>}[uull]^{\widetilde \lambda}}\]
\caption{A commutative diagram\label{unnamed}}
\end{figure}
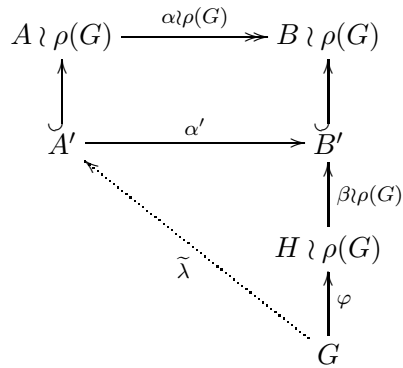
Since $G$
is $m$-quasifree, there exists   a continuous   epimorphism
$\widetilde \lambda\colon  G\rightarrow A'$ such that $\alpha'
\widetilde \lambda = (\beta\wr \rho (G))\varphi$. Then,  for each
$g\in G$, $\widetilde \lambda(g)= (\widetilde f_g, \rho(g))$, for some
$\widetilde f_g\in A^\Sigma$.

Let $T=\{t_1=1, t_2, \dots , t_k\}$ be a right transversal of $H$ in $G$.  For $i=1, \dots, k$, define    $\lambda_i\colon  H^{t_i}\rightarrow A$ to be $\lambda_i= \pi_{A,i}\widetilde \lambda|_{H^{t_i}}$, i.e., $\lambda_i (x)=\widetilde f_x(Ht_i)$, for $x\in H^{t_i}$.
According to Lemma~\ref{naturality}, the diagram
 \[\xymatrix{ &A\ar@{->>}[rr]^\alpha &&B \\      A\wr \rho (H^{t_i})\ar@{->>}[rr]^{\alpha\wr \rho (H^{t_i})}\ar@{->>}[ru]^{\pi_{A,i}}&& B\wr \rho(H^{t_i})\ar@{->>}[ru]^{\pi_{B,i}}\\                                              A^\Sigma\ar@{^{(}->}[u]^\iota&&H\wr \rho (H^{t_i})\ar@{->>}[r]^(.6){\pi_{H,i}}\ar[u]_{\beta\wr \rho(H^{t_i})}&H\ar@{->>}[uu]_\beta\\    \\                                                    &&H^{t_i}\ar[uu]_(.6){\varphi|_{H^{t_i}}}\ar[uuull]_(.6){\widetilde \lambda|_{H^{t_i}}}\ar[uur]_{\mathrm {inn}_{t_i}|_{H^{t_i}}} \\  &&H_G\ar@{_{(}->}[u]\ar[uuull]^{\widetilde\lambda|_{H_G}}} \]
commutes.  Thus $\beta\circ \mathrm{inn}_{t_i}|_{H^{t_i}}= \alpha\lambda_i.$

We claim that $\lambda_i$ is surjective.  Let $a\in A$ and let
$b=\alpha(a)$.  Since $\beta$ is surjective, the commutativity of
the above diagram ensures that there exists $(f,\rho(x)))\in
(B\wr \rho(H^{t_i}))\cap B'$, where $x\in H^{t_i}$, with
$f(Ht_i)=b$.  Choose  $f'\colon \Sigma\to A$ to be  any function so
that $f'(Ht_i)=a$ and $\alpha f'=f$; then $(f',\rho(x))\in A'$.
Therefore, $\pi_{A,i}$ takes $A'\cap (A\wr \rho(H^{t_i}))$ onto
$A$. Because \mbox{$\ker \widetilde \lambda\le \ker \rho=H_G \le
H^{t_i}$}, it follows that  $\widetilde \lambda(g)\in A\wr
\rho(H^{t_i})$ implies $g\in H^{t_i}$.  We deduce that  $\widetilde
\lambda|_{H^{t_i}}\colon H^{t_i}\to A'\cap (A\wr \rho(H^{t_i}))$ is
an epimorphism, and hence so is $\lambda_i$, proving the claim.

Since $G$ is quasifree, the total number of epimorphisms $\widetilde
\lambda\colon G\rightarrow A'$  such that $ \alpha' \widetilde \lambda
= (\beta\wr \rho (G))\varphi$ is $m$.   Since $H_G=
\bigcap_{i=1}^kH^{t_i}$ has finite index in $G$,  these   $\widetilde
\lambda$ restrict to $m$ different homomorphisms
\[\widetilde\lambda|_{H_G} \colon H_G\rightarrow A\wr\rho(H_G)=
A^\Sigma.\]
Recalling from Lemma~\ref{naturality} that the $\pi_{A,i}\colon
A^{\Sigma}\to A$ ($i=1,\ldots, k$) are the direct product
projections, we conclude that $\widetilde\lambda|_{H_G}$ is determined
by the maps
$\pi_{A,i}\widetilde\lambda|_{H_G}=\lambda_i|_{H_G}$, $i=1, \dots ,
k$. It follows that there exists some $j\in \{1, \dots , k\}$, such
that the number of different maps $\lambda_j|_{H_G}$ constructed in
this manner
is precisely $m$.

For each of these $\lambda_j$, define $\lambda= \lambda_j\circ
\mathrm{inn}_{t_j^{-1}}|_{H}$.  Then, since $H_G$ has finite index in
$H$, we have constructed  $m$
different epimorphisms $\lambda\colon H\rightarrow A$  such that
$\alpha\lambda = \beta$. Finally, observe that there cannot be more
such $\lambda$ since  the minimal number $d(H)$ of generators of $H$
converging to $1$ is $m$  and $A$ is finite.  This completes the
proof.
\end{proof}

\bibliographystyle{amsplain}
\bibliography{standard2}

\end{document}